\documentclass[psamsfonts]{amsart}

\usepackage{amssymb,amsfonts}
\usepackage[all,arc]{xy}
\usepackage{enumerate}
\usepackage{mathrsfs}
\usepackage[utf8]{inputenc}
\usepackage{hyperref}
\usepackage{mathtools}
\usepackage{tikz}
\usepackage{enumitem}
\newtheorem{thm}{Theorem}[section]
\newtheorem{theorem}{Theorem}[section]
\newtheorem{cor}[thm]{Corollary}

\newtheorem{lemma}[thm]{Lemma}

\newtheorem{question}[thm]{Question}
\newtheorem{claim}[thm]{Claim}

\theoremstyle{definition}
\newtheorem{definition}[thm]{Definition}

\theoremstyle{remark}
\newtheorem{remark}[thm]{Remark}

\newtheorem{observation}[thm]{Observation}

\makeatletter
\let\c@equation\c@thm
\makeatother
\numberwithin{equation}{section}

\title[Some remarks on uncountable rainbow Ramsey theory]{Some remarks on uncountable rainbow Ramsey theory}

\author{Jing Zhang}

\newcommand{\Addresses}{{
  \bigskip
  \footnotesize

 \textsc{Department of Mathematics,\\ Bar-Ilan University, Ramat-Gan 5290002, Israel}\par\nopagebreak
  \textit{E-mail}: \texttt{jingzhan@alumni.cmu.edu}
}}

\begin{document}

\begin{abstract}
We discuss the rainbow Ramsey theorems at limit cardinals and successors of singular cardinals, addressing some questions in \cite{MR2354904} and \cite{MR2902230}. In particular, we show for inaccessible $\kappa$, $\kappa\to^{poly}(\kappa)^2_{2-bdd}$ does not characterize weak compactness and for singular $\kappa$, $\mathrm{GCH}$ implies $\kappa^+\not\to^{poly} (\eta)^2_{<\kappa-bdd}$ for any $\eta\geq cf(\kappa)^+$ and $\square_\kappa$ implies $\kappa^+\to^{poly} (\nu)^2_{<\kappa-bdd}$ for any $\nu<cf(\kappa)^+$.
We also provide a simplified construction of a model for $\omega_2\not\to^{poly} (\omega_1)^2_{2-bdd}$ originally constructed in \cite{MR2902230} and show the witnessing coloring is indestructible under strongly proper forcings but destructible under some c.c.c forcing. Finally, we conclude with some remarks and questions on possible generalizations to rainbow partition relations for triples. 
\end{abstract}

\maketitle
\let\thefootnote\relax\footnotetext{ 

I thank Uri Abraham, James Cummings, Assaf Rinot and Ernest Schimmerling for helpful discussions, comments and corrections on earlier drafts. I am grateful to the anonymous referee who provides extensive comments and corrections which greatly improve the exposition.

The work was done when I was a graduate student at Carnegie Mellon University supported in part by the US tax payers. Part of the revision was done when I am a post doctoral fellow in Bar-Ilan University, supported by the Foreign Postdoctoral Fellowship Program of the Israel Academy of Sciences and Humanities and by the Israel Science Foundation (grant agreement 2066/18).
}

\section{Introduction}

Fix ordinals $\lambda,i, \kappa$ and $n\in \omega$.

\begin{definition}
 We use $\lambda\to (\kappa)^n_i$ to abbreviate: for any $f: [\lambda]^n \to i$, there exists $A\subset \lambda$ of order type $\kappa$ such that $f\restriction [A]^n$ is a constant function. Such $A$ is called a \emph{monochromatic} subset of $\lambda$ (with respect to $f$).
\end{definition}

\begin{definition}
We use $\lambda\to^{poly} (\kappa)^n_{i-bdd}$ to abbreviate: for any $f: [\lambda]^n \to \lambda$ that is \emph{i-bounded}, namely for any $\alpha\in \lambda$, $|f^{-1}\{\alpha\}|\leq i$, there exists $A\subset \lambda$ of order type $\kappa$ such that $f\restriction [A]^n$ is injective. Such $A$ is called a \emph{rainbow} subset of $\lambda$ (with respect to $f$).
\end{definition}

\begin{remark}
$\to^{poly}$ is sometimes denoted as $\to^*$. We adopt $\to^{poly}$ to avoid possible confusion, as rainbow subsets are sometimes called ``polychromatic'' subsets.
\end{remark}

$\lambda\to (\kappa)^n_i$ implies $\lambda\to^{poly} (\kappa)^n_{i-bdd}$ as given a $i$-bounded coloring it is possible to cook up a dual $i$-coloring for which any monochromatic subset will be a rainbow subset for the original coloring. This is the \emph{Galvin's trick}. This explains why rainbow Ramsey theory is also called sub-Ramsey theory in finite combinatorics.

In many cases, the rainbow analogue is a strict weakening. For example: 

\begin{enumerate}

\item[1] In finite combinatorics, the sub-Ramsey number $sr(K_n, k)$, which is the least $m$ such that $m\to^{poly}(n)_{k-bdd}^2$, is bounded by a polynomial in $n$ and $k$ (Alspach, Gerson, Hahn and Hell \cite{MR867747}). This is in contrast with the Ramsey number which grows exponentially.

\item[2] In reverse mathematics, over $RCA_0$, $\omega\to^{poly}(\omega)^2_{2-bdd}$ does not imply $\omega\to(\omega)^2_2$ (Csima and Mileti \cite{MR2583822}).

\item[3] In combinatorics on countably infinite structures, the Rado graph is Rainbow Ramsey but not Ramsey (Dobrinen, Laflamme, and Sauer \cite{MR3518438}).

\item[4] In combinatorics on the ultrafilters on $\omega$, Martin's Axiom implies there exists a Rainbow Ramsey ultrafilter that is not a Ramsey ultrafilter (Palumbo \cite{MR3135506}).

\item[5] In uncountable combinatorics, ZFC proves $\omega_1\not\to (\omega_1)_2^2$ but $\omega_1\to^{poly} (\omega_1)_{2-bdd}^2$ is consistent with ZFC (Todorcevic \cite{MR716846}). 

\end{enumerate}

Results in this note serve as further evidence that rainbow Ramsey theory is a strict weakening of Ramsey theory. We focus on the area of uncountable combinatorics.

The organization of the paper is: 

\begin{enumerate}
\item In Section \ref{inaccessible}, we discuss rainbow Ramsey theorems at limit cardinals. In particular, we show $\kappa\to^{poly} (\kappa)^2_{2-bdd}$ for an inaccessible cardinal $\kappa$ does not imply $\kappa$ is weakly compact, answering a question in \cite{MR2354904};
\item In Section \ref{singular}, we discuss the rainbow Ramsey theorems at the successor of singular cardinals. Answering a question in \cite{MR2902230}, we show $\mathrm{GCH}+\square_\kappa$ implies $\kappa^+\not\to^{poly} (\eta)^2_{<\kappa-bdd}$ for any $\eta\geq cf(\kappa)^+$ and $\kappa^+\to^{poly} (\nu)^2_{<\kappa-bdd}$ for any $\nu<cf(\kappa)^+$ .
\item In Section \ref{indestructible}, we use the method of Neeman developed in \cite{MR3201836} to simplify the construction of a model by Abraham and Cummings \cite{MR2902230} in which $\omega_2\not\to^{poly} (\omega_1)_{2-bdd}^2$. Furthermore, we show in this model, the witnessing coloring is indestructible under strongly proper forcings but destructible under c.c.c forcings. In other words, the coloring witnessing $\omega_2\not\to^{poly} (\omega_1)_{2-bdd}^2$ remains the witness to the same negative partition relation in any strongly proper forcing extension but there exists a c.c.c forcing extension that adds a rainbow subset of size $\omega_1$ for that coloring.
As a result, $\omega_2\not\to^{poly} (\omega_1)_{2-bdd}^2$ is compatible with the continuum being arbitrarily large.
\item In Section \ref{generalization}, we briefly discuss possibilities and restrictions of generalizations to partition relations for triples.
\end{enumerate}

\section{Rainbow Ramsey at limit cardinals}\label{inaccessible}

In \cite{MR2354904}, Abraham, Cummings and Smyth studied the rainbow Ramsey theory at small uncountable cardinals and successors of regular cardinals. They asked what can be said about the rainbow Ramsey theory at inaccessible cardinals. A test question they asked was for any inaccessible cardinal $\kappa$, whether $\kappa\to^{poly}(\kappa)^2_2$ characterize weak compactness. We answer this in the negative.
Fix a regular uncountable cardinal $\kappa$.

\begin{definition}
We say $f: [\kappa]^n \to \kappa$ is a normal coloring if whenever $\bar{a}, \bar{b}\in [\kappa]^n$ are such that $f(\bar{a})=f(\bar{b})$, then $\max \bar{a}=\max \bar{b}$.
\end{definition}

\begin{definition}
A normal function $f:[\kappa]^2\to \kappa$ is regressively bounded (reg-bdd) if there exists $\lambda<\kappa$ such that $\kappa \cap cof(\geq \lambda)$ is stationary in $\kappa$ and for all $\alpha\in \kappa\cap cof(\geq \lambda)$, and $i<\kappa$, $\{\beta\in \alpha: f(\beta, \alpha)=i\}$ is bounded in $\alpha$.
We use $\kappa \to^{poly} (\kappa)^2_{reg-bdd} $ to denote the statement: for any normal regressively bounded $f: [\kappa]^2\to \kappa$, there exists a subset $A\in [\kappa]^\kappa$ such that $A$ is a rainbow subset for $f$.
\end{definition}

\begin{remark}
Notice for any weakly inaccessible cardinal $\kappa$ and any cardinal $\lambda<\kappa$, $\kappa \to^{poly} (\kappa)^2_{reg-bdd}$ implies $\kappa \to^{poly} (\kappa)^2_{\lambda-bdd}$. To see this, given $f: [\kappa]^2\to \kappa$ that is $\lambda$-bounded, recursively, we may find a subset $B\in [\kappa]^\kappa$ such that $f\restriction [B]^2$ is normal. Hence without loss of generality we may assume $f$ is normal. Then it is easy to see that $f$ is regressively bounded witnessed by $\lambda^+$.
\end{remark}

Even though we cannot employ Galvin's trick of dual colorings since 
there may not be any $\lambda<\kappa$ that bounds the sizes of color classes, we do have that if $\kappa$ is weakly compact, then $\kappa\to^{poly} (\kappa)_{reg-bdd}^2$.

It turns out that weak compactness is not necessary. It suffices when $\kappa$ is a ``generic large cardinal'' (for more on this topic, see \cite{MR2768692}). Recall that a \emph{$\overrightarrow{C}$-sequence} on $\kappa$ is $\langle C_\alpha: \alpha\in \lim \kappa\rangle$ such that each $C_\alpha\subset \alpha$ is a club subset of $\alpha$.

\begin{lemma}\label{vanilla}
Let $\kappa$ be a regular cardinal.
Consider the following statements: 
\begin{enumerate}
\item Every $\overrightarrow{C}$-sequence on $\kappa$ is \emph{trivial stationarily often}, namely, there exists a club $D\subset \kappa$ such that for any $\alpha<\kappa$, there exist stationarily many $\beta<\kappa$ such that $D\cap \alpha \subset C_\beta$;
\item $\kappa\to^{poly} (\kappa)^2_{reg-bdd}$;
\item every $\overrightarrow{C}$-sequence on $\kappa$ is \emph{trivial}, in the sense that there exists a club $D\subset \kappa$ such that for any $\alpha<\kappa$ there exists $\beta<\kappa$ such that $D\cap \alpha\subset C_\beta$. 
\end{enumerate}
Then (1) implies (2).
\end{lemma}

\begin{proof}
We first prove (1) implies (2).
Given a regressively bounded $f$ witnessed by $\lambda<\kappa$, we may assume $f(\cdot, \delta): \delta\to \delta$. It is not hard to see that (1) implies that $\kappa$ is a limit cardinal. Hence we may also assume $\lambda\geq \aleph_1$.

For each $\alpha\in cof(\geq \lambda)\cap \kappa$, we let $C_\alpha\subset \alpha$ be a club of order type $cf(\alpha)$ such that for any $\gamma_0<\gamma_1\in C_\alpha$, $f(\gamma_0, \alpha)\neq f(\gamma_1, \alpha)$. We can achieve this by first picking a club subset $C'_\alpha\subset \alpha$ of order type $cf(\alpha)$, and $C_\alpha\subset C_\alpha'$ is the closure points for the following function: $f: C'_\alpha \to C'_\alpha$ such that $f(\beta)$ is the least $\beta'$ such that for all $\beta''\geq \beta'$, $f(\beta, \alpha)\neq f(\beta'', \alpha)$. Note $f$ is well-defined by the regressive bounding condition on $f$ and we make sure that $f(\cdot, \alpha)\restriction C_\alpha$ is injective. If $\alpha\in cof(<\lambda)\cap \kappa$, then just let $C_\alpha$ be any club of type $cf(\alpha)$.

Let $D\subset \kappa$ be a club given by the conclusion of (1). We may then build a continuous sequence $\langle \alpha_i\in D: i<\kappa\rangle$ such that for any $i<\kappa$, $D\cap \alpha_i\subset C_{\alpha_{i+1}}$. Notice that on a tail of this sequence, $\alpha_{i+1}\in cof(\geq \lambda)\cap \kappa$. We may without loss of generality assume that $\alpha_{i+1}\in cof(\geq \lambda)\cap \kappa$ for all $i<\kappa$.

Let $Y\subset \kappa$ be a set of size $\kappa$ such that for any $i<j\in Y$, $\alpha_{i+1} < \alpha_j$. We claim that $X=_{def} \{\alpha_{i+1}: i\in Y\}$ is a rainbow subset for $f$. Fix $i<j<k \in Y$, since $\alpha_{i+1}<\alpha_{j+1}\in D\cap \alpha_{k} \subset C_{\alpha_{k+1}}$, we know that $f(\alpha_{i+1}, \alpha_{k+1})\neq f(\alpha_{j+1},\alpha_{k+1})$.

We prove (2) implies (3).

Let a $\langle C_\alpha: \alpha<\kappa\rangle$ be a given. Define $f(\alpha,\beta)=(\min (C_\beta - (\alpha+1)),\beta)$. It can be easily checked that that $f$ is regressively bounded. Let $X\subset \kappa$ be a rainbow subset of size $\kappa$ and let $C=\lim X=\{\delta \in \lim \kappa: \sup X\cap \delta=\delta\}$. We claim that $C$ is what we are looking for. It clearly suffices to show that any $\alpha\in C$ and any $\beta\in X- (\alpha+1)$, it is true that $\alpha\in C_\beta$. Since $\alpha$ is a limit point of $X$, and $f(\cdot, \beta)\restriction X\cap \alpha$ is injective, it must be the case that $\alpha\in \lim C_\beta$, which implies that $\alpha\in C_\beta$.

\end{proof}

\begin{remark}
Lambie-Hanson and Rinot \cite{CLHRinot} introduced and studied a cardinal invariant on $\overrightarrow{C}$-sequences. The statement (3) in Lemma \ref{vanilla} is what they call $\chi(\kappa)\leq 1$. Furthermore, they show $\chi(\kappa)\leq 1$ implies that $\kappa$ is greatly Mahlo.
\end{remark}

%
%
%

\begin{definition}
Suppose $\kappa$ is regular and ${}^{<\kappa}\kappa=\kappa$. We say \emph{$\kappa$ is generically weakly compact via $\kappa$-c.c forcings} if for any transitive $M$ with $|M|=\kappa$, $\kappa\in M$, ${}^{<\kappa}M\subset M$, there exists a $\kappa$-c.c forcing $P$ such that for any generic $G\subset P$ over $V$, in $V[G]$, there exists an elementary embedding $j: M\to N$ where $N$ is transitive and $\mathrm{crit}(j)=\kappa$.
\end{definition}

\begin{thm}\label{saturatedproof}
If a regular cardinal $\kappa$ is generically weakly compact via $\kappa$-c.c forcings, then $\kappa\to^{poly} (\kappa)_{reg-bdd}^2$.
\end{thm}

\begin{proof}
We will show that every $\overrightarrow{C}$-sequence on $\kappa$ is trivial stationarily often, then the theorem follows from Lemma \ref{vanilla}. Given a $\overrightarrow{C}$-sequence on $\kappa$, $\bar{C}=\langle C_\alpha: \alpha<\kappa\rangle$, let $X\prec H(\theta)$ containing $\bar{C}$ where $\theta$ is a large enough regular cardinal, such that $|X|=\kappa$, $\kappa\subset X$ and ${}^{<\kappa}X\subset X$. Let $\pi: X\to M$ be the transitive collapse. Note that $\bar{C}$ is fixed by $\pi$. By the assumption on $\kappa$, there is some generic extension $V[G]$ by a $\kappa$-c.c forcing such that there exists an embedding $j: M\to N$ with $N$ being transitive and $\mathrm{crit}(j)=\kappa$.

Let $D'=j(\bar{C})(\kappa)$. Then $D'\subset \kappa$ is a club set in $V[G]$. Since the forcing is $\kappa$-c.c, there exists a club $D\subset \kappa\in V$ such that $D\subset D'$. We claim that $D$ trivializes $\bar{C}$ stationarily often. For any $\alpha<\kappa$, $D\cap \alpha\in M$ since $M$ contains $\kappa$ and is closed under $<\kappa$-sequences in $V$. Let $S=\{\beta<\kappa: D\cap \alpha\subset C_\beta\}$. We will show that $S$ is stationary. Notice that $S\in M$ and $M\models S$ is a stationary subset of $\kappa$. To see this, suppose for the sake of contradiction there is a club $D^*\in M$ disjoint from $S$. Then $\kappa \in j(D^*)$ since $j(D^*)$ is a club in $j(\kappa)$ by elementarity and $j(D^*)\cap \kappa=D^*$ which is unbounded in $\kappa$. This implies $j(D)\cap \kappa=D\not \subset j(\bar{C})(\kappa)=D'$, which is a contradiction. 
Since $\pi^{-1}(S)=S$, we know that $X\models S$ is a stationary subset of $\kappa$, hence by elementarity, $S\subset \kappa$ is stationary.

\end{proof}

\begin{remark}

To get a model of a cardinal $\kappa$ which is generically weakly compact via $\kappa$-c.c forcings but not weakly compact, we can proceed by the following: first we prepare the ground model such that the weakly compact cardinal $\kappa$ is indestructible under $\mathrm{Add}(\kappa,1)$, and then use a theorem of Kunen (\cite{MR495118}) that $\mathrm{Add}(\kappa,1)$ is forcing equivalent to $P*\dot{T}$ where $P$ adds a homogeneous $\kappa$-Suslin tree $\dot{T}$. The final model will be $V^P$.
\end{remark}

\begin{remark}
The Kunen model also shows that the existence of a $\kappa$-Suslin tree is consistent with $\kappa\to^{poly} (\kappa)^2_{reg-bdd}$. The existence of a $\kappa$-Suslin tree is sometimes strong enough to refute some weak consequences of $\kappa\to (\kappa)^2_2$. For example  Todorcevic proved in \cite{Todorcevic1989-TODTSA-5} that for any regular uncountable cardinal $\kappa$, the existence of $\kappa$-Suslin tree implies $\kappa\not\to [\kappa]^2_\kappa$, namely there exists a coloring $f: [\kappa]^2\to \kappa$ such that any $X\in [\kappa]^\kappa$, $f'' [X]^2 =\kappa$. 
\end{remark}

\begin{cor}
It is consistent relative to a weakly compact cardinal that for some inaccessible cardinal $\kappa$ that is not weakly compact, $\kappa\to^{poly} (\kappa)_{\lambda-bdd}^2$ for any $\lambda<\kappa$.
\end{cor}

\begin{cor}
If $\kappa$ is real-valued measurable, then $\kappa\to^{poly} (\kappa)_{\lambda-bdd}^2$ for any $\lambda<\kappa$.
\end{cor}

\begin{cor}
If $\kappa$ is weakly compact, then $\kappa\to^{poly} (\kappa)^2_{reg-bdd}$ is indestructible under any forcing satisfying $\lambda$-c.c. for some $\lambda<\kappa$.
\end{cor}

The trick of using some large enough ordinal to ``guide'' the construction can also be used analogously to prove the following, which provides more contrast with its dual Ramsey statement:

\begin{lemma}\label{singularstronglimit}
For any singular strong limit $\kappa$, $\kappa\to^{poly} (\kappa)_{\lambda-bdd}^2$ for any $\lambda<\kappa$. 
\end{lemma} 

\begin{remark}
Given a $\lambda$-bounded coloring $f$ on $[\kappa]^2$, we claim that there is $B\in [\kappa]^{\kappa}$ such that $f\restriction [B]^2$ is normal. Fix a continuous sequence of strictly increasing regular cardinals $\langle \kappa_i: i<cf(\kappa)\rangle$ with $\kappa_0>\max \{cf(\kappa), \lambda\}$ converging to $\kappa$. We find $\langle A_i: i<cf(\kappa)\rangle$ such that 
\begin{itemize}
\item for any $i<cf(\kappa)$, $A_i\subset \kappa_i$ and $|A_i|=\kappa_i$
\item for any $i<j<cf(\kappa)$, $A_i\subsetneq A_j$
\item for any limit $\delta<cf(\kappa)$, $A_\delta=\bigcup_{i<\delta} A_i$
\item for any $i<cf(\kappa)$, $f\restriction [A_i]^2$ is normal
\end{itemize}
The construction clearly gives $B=\bigcup_{i<cf(\kappa)} A_i$ such that $f\restriction [B]^2$ is normal.
The construction at limit stages is clear. At stage $i+1$, we inductively find a subset $C\subset \kappa_{i+1}-\kappa_i$ of size $\kappa_{i+1}$ such that $f\restriction [A_i\cup C]^2$ is normal. Suppose we have built $C'\subset \kappa_{i+1}-\kappa_i$ of size $\leq \kappa_i$, we demonstrate how to add one more element. As $|C'\cup A_i|\leq \kappa_i, \lambda<\kappa_i$ and $\kappa_{i+1}$ is regular, there exists $\gamma>\max C'+1$ such that there do not exist $a\in [C'\cup A_i]^2, \beta\in C'\cup A_i$ with $f(a)=f(\beta, \gamma)$. It is easy to see that $f\restriction [A_i\cup C'\cup \{\gamma\}]^2$ is normal.
\end{remark}

\begin{proof}[Proof of Lemma \ref{singularstronglimit}]
Fix a $\lambda$-bounded coloring $f: [\kappa]^2\to \kappa$. By the remark above, we may assume $f$ is normal.
Let $\eta=cf(\kappa)$. Fix an increasing sequence of regular cardinals $\langle \kappa_i: i<\eta\rangle$ such that 
\begin{enumerate}
\item $\kappa_0>\max\{\lambda, \eta\}$;
\item $\langle \kappa_i: i<\eta\rangle$ converges to $\kappa$;
\item $\kappa_{i+1}^{\kappa_{i}}=\kappa_{i+1}$ for all $i<\eta$.
\end{enumerate}

Let $\theta$ be a large enough regular cardinal and fix an $\in$-increasing chain $\langle N_i\prec H(\theta): i<\eta \rangle$ such that $|N_i|=\kappa_{i+1}$, $\kappa_{i+1}\subset N_i$, $\sup (N_i\cap\kappa_{i+2})=_{def} \delta_i\in \kappa_{i+2}\cap cof(\kappa_{i+1})$, ${}^{\kappa_i} N_i\subset N_i$.
We arrange that $\lambda,f, \langle \kappa_i : i<\eta\rangle \in N_0$.

We will recursively build $\langle A_i: i<\eta\rangle$ such that $A_i\subset N_i\cap \kappa_{i+2}$ and $|A_i|=\kappa_{i}^+$ satisfying:

\emph{for all $j\geq i$, $A_i\cup \{\delta_j\}$ is a rainbow subset of $f$.}

Recursively, suppose $A_k\subset N_k\cap \kappa_{k+2}$ for $k<i$ have been built. Let $A^*=\bigcup_{k<i} A_k \subset \kappa_{i+1}\subset N_i$. Notice that $|A^*|\leq \kappa_{i}$. We will enlarge $A^*$ with $\kappa_{i}^+$ many elements in $\delta_i-\kappa_{i+1}$. More precisely, we will find $C=\{\alpha_k\in \delta_i -\kappa_{i+1}  :  k<\kappa_{i}^+\}$ such that $A^*\cup C \cup \{\delta_j\}$ is a rainbow subset of $f$ for all $j\geq i$. We finish by setting $A_i=A^*\cup C$.

Suppose we have built $C_\nu=\{\alpha_k: k<\nu\}$ for some $\nu<\kappa_{i}^+$ satisfying the requirement. Since ${}^{\kappa_i} N_i\subset N_i$, we have $A^*\cup C_\nu \in N_i$.
Let $A(A^*\cup C_\nu)=_{def}\{\gamma<\kappa_{i+2}: A^*\cup C_\nu \cup \{\gamma\} \text{ is a rainbow subset for }f\}$. Since $A(A^*\cup C_\nu)\in N_i$ and $\delta_i\in A(A^*\cup C_\nu)$,  we know that $A(A^*\cup C_\nu)$ is a stationary subset of $\kappa_{i+2}$.

Let $B_j=_{def} \{\rho\in A(A^*\cup C_\nu): \exists \alpha\in A^*\cup C_\nu \ f(\alpha, \delta_j)=f(\rho, \delta_j)\}$ for each $j\geq i$. As $|A^*\cup C_\nu| \leq \kappa_{i}$ and the coloring is $\lambda$-bounded, we know that $|B_j|\leq \kappa_i$ for any $j\geq i$. Pick any $\gamma \in A(A^*\cup C_\nu)- \bigcup_{i\leq  j<\eta} B_j$ with $\gamma > \max A^*\cup C_\nu$. We claim that this $\gamma$ is as desired, namely $A^*\cup C_\nu\cup \{\gamma\}\cup \{\delta_j\}$ is a rainbow subset for all $j\geq i$. Indeed, fix some $j\geq i$. By the fact that $\gamma,\delta_j\in A(A^*\cup C_\nu)$, the only bad possibility is that for some $\alpha\in A^*\cup C_\nu$, $f(\alpha,\delta_j)=f(\gamma,\delta_j)$. But this is ruled out by the fact that $\gamma\not\in B_j$.

\end{proof}

\begin{remark}\label{Indestructible}
We can strengthen the conclusion of Lemma \ref{singularstronglimit} to that $\kappa\to^{poly} (\kappa)_{\lambda-bdd}^2$ for any $\lambda<\kappa$ and it remains true in any forcing extension satisfying $<\gamma$-covering property (see Definition \ref{covering}) for some cardinal $\gamma<\kappa$. The proof is similar to that of Theorem \ref{CountableCase}. Hence it is also possible for a singular cardinal which is not a strong limit to satisfy the conclusion of Lemma \ref{singularstronglimit}.
\end{remark}

\begin{remark}\label{FreeSetTheorem}
The following strengthening is also true: if $\kappa$ is a strong limit singular cardinal, then $\kappa\to^{poly} (\kappa)^n_{\lambda-bdd}$ for any $\lambda<\kappa$ and $n\in \omega$. This is an immediate consequence of the following theorem (Theorem 45.4 in \cite{MR795592}): given a strong limit singular cardinal $\kappa$ and some $\lambda<\kappa$, we have that for any $f: [\kappa]^n \to [\kappa]^{\lambda}$, there exists $H\subset \kappa$ of cardinality $\kappa$ such that for any $x\in [H]^2$, $f(x)\cap (H-x) =\emptyset$ (such $H$ is called \emph{$f$-free}). To see the implication, given a $g: [\kappa]^2\to \kappa$ which is $\lambda$-bounded, consider $f: [\kappa]^2\to [\kappa]^\lambda$ defined as $f(x)=\bigcup \{y: g(y)=g(x)\}-x$. We leave it to the reader to verify that any $f$-free set is a rainbow subset for $g$. However, the proof of Theorem 45.4 in \cite{MR795592} heavily uses the Erd\H{o}s-Rado theorem, it is thus not clear if the proof can be generalized to give the forcing indestructibility result as in Remark \ref{Indestructible}. We decide to keep the proof of a weaker result that entails generalizations.
\end{remark}

\begin{question}
If an inaccessible $\kappa$ carries a non-trivial $\kappa$-complete $\kappa$-saturated normal ideal, is it true that $\kappa\to^{poly}(\kappa)^n_{\lambda-bdd}$ for all $n\in \omega$ and all $\lambda<\kappa$?
\end{question}

\section{The extent of Rainbow Ramsey theorems at successors of singular cardinals}\label{singular}

In \cite{MR2354904} and \cite{MR2902230}, it is shown that if $\mathrm{GCH}$ holds, then $\kappa^+\to^{poly} (\eta)^2_{<\kappa-bdd}$ for any regular cardinal $\kappa$ and ordinal $\eta<\kappa^+$ and moreover the partition relations continue to hold in any $\kappa$-c.c. forcing extension. The authors ask what we can say when $\kappa$ is singular. We will address this question by showing $\mathrm{GCH}$ implies $\kappa^+\to^{poly} (\eta)^2_{<\kappa-bdd}$ for all $\eta<cf(\kappa)^+$ and $\square_\kappa$ implies $\kappa^+\not\to^{poly} (\eta)^2_{<\kappa-bdd}$ for all $\eta\geq cf(\kappa)^+$. For the latter, as we will see below, a weaker hypothesis suffices.

\begin{observation}\label{cofinalityObservation}
If $\kappa$ is singular of cofinality $\lambda<\kappa$, then $\kappa^+\not \to^{poly} (\lambda^+ +1)^2_{<\kappa-bdd}$.
\end{observation}

\begin{proof}
For each $\beta\in \kappa^+$, fix disjoint $\{A_{\beta, n}: n\in \lambda\}$ such that each set has size $<\kappa$ and $\bigcup_{n\in \lambda} A_{\beta, n} =\beta$.
Define a coloring by mapping $\{\alpha,\beta\}\in [\kappa^+]^2\mapsto (n,\beta)$ if $n$ is the unique element in $\lambda$ that $\alpha\in A_{\beta,n}$. This coloring is easily seen to be $<\kappa$-bounded.  For any subset $A$ of order type $\lambda^++1$, let $\delta$ be the top element. Now by pigeon hole, there exists $n\in \lambda$, such that $|A\cap A_{\delta,n}|\geq \lambda^+$. For any $\alpha<\beta\in A\cap A_{\delta,n}$, $f(\alpha,\delta)=(n,\delta)=f(\beta,\delta)$. Thus $A$ is not a rainbow subset.
\end{proof}

The following connects the rainbow partition relations with sets in Shelah's approachability ideal. Fix a singular cardinal $\kappa$ with cofinality $\lambda<\kappa$ for Definitions \ref{sing1} and \ref{sing2}.

\begin{definition}\label{sing1}
A set $S\subset \kappa^+$ is in $I[\kappa^+; \kappa]$ iff there is a sequence $\bar{a}=\langle a_\alpha \in [\kappa^+]^{<\kappa} : \alpha<\lambda\rangle$ and a closed unbounded $C\subset \kappa^+$ such that for any $\delta\in C\cap S$ is singular and weakly approachable with respect to the sequence $\bar{a}$, namely there is an unbounded $A\subset \delta$ of order type $cf(\delta)$ such that any $\alpha<\delta$ there exists $\beta<\delta$ with $A\cap \alpha\subset a_\beta$.
\end{definition}

Notice that $I[\kappa^+; \kappa]$ contains $I[\kappa^+]$, which is Shelah's approachability ideal. For more details on these matters, see \cite{MR2768694}.

\begin{definition}[Definition 3.24, 3.25 \cite{MR2768694}]\label{sing2}
$d: [\kappa^+]^2\to cf(\kappa)$ is 

\begin{enumerate}
\item \emph{normal} if 
$$ i<cf(\kappa) \rightarrow \sup_{\alpha<\kappa^+} |\{\beta<\alpha: d(\beta,\alpha)<i\}|<\kappa,$$
\item \emph{transitive} if for any $\alpha<\gamma<\beta<\kappa^+$, $d(\alpha,\beta)\leq \max \{d(\alpha,\gamma), d(\gamma,\beta)\}$,
\item \emph{approachable} on $S\subset \lim \kappa^+$ if for any $\delta\in S$, there is a cofinal $A\subset \delta$ such that for any $\alpha\in A$, $\sup \{d(\beta,\alpha): \beta\in A\cap \alpha\}<cf(\kappa)$.
\end{enumerate}

\end{definition}

It is a consequence of Theorem 3.28 in \cite{MR2768694} that $\kappa^+\cap cof(\lambda^+)\in I[\kappa^+; \kappa]$ implies the existence of a normal $d$ that is approachable on $E\cap \kappa^+\cap cof(\lambda^+)$ for some club $E\subset \kappa^+$.

\begin{lemma}
For a singular cardinal $\kappa$ with cofinality $\lambda<\kappa$, we have that $\kappa^+\cap cof(\lambda^+)\in I[\kappa^+; \kappa]$ implies $\kappa^+\not\to^{poly} (\lambda^+)^2_{<\kappa-bdd}$.
\end{lemma}

\begin{proof}
Fix a normal $d$ that is approachable at $E\cap \kappa^+\cap cof(\lambda^+)$ for some club $E\subset \kappa^+$. Define $f: [ E]^2\to \kappa^+$ such that $f(\alpha,\beta)=(d(\alpha,\beta),\beta)$. The normality of $d$ implies $f$ is $<\kappa$-bounded. Given $A\in [E]^{\lambda^+}$, let $\gamma=\sup A$. Then $d$ is approachable at $\gamma$. Fix some unbounded $B\subset \gamma$ of order type $\lambda^+$ witnessing the approachability of $d$. We may assume there exists $\eta_0<\lambda$ such that $\sup d''[B]^2 \leq \eta_0$. To see why we can do this, note that by the approachability assumption on $d$, we know for each $\alpha\in B\cap \gamma$, $\eta_\alpha'=\sup \{d(\beta, \alpha): \beta\in B\cap \alpha\}<\lambda$. Find $B'\in [B]^{\lambda^+}$ and $\eta_0<\lambda$ such that for any $\alpha\in B'$, $\eta_\alpha'=\eta_0$. It is clear that $\sup d''[B']^2 \leq \eta_0$. Without loss of generality, we may assume $B'=B$.

Pick the following increasing sequences $\langle a_i\in A: i<\lambda^+\rangle$ and $\langle b_i \in B : i<\lambda^+\rangle$ satisfying that for all $i<\lambda^+$, $b_i<a_i<b_{i+1}$. By the Pigeon Hole principle, we can find $D\in [\lambda^+]^{\lambda^+}$ and some $\eta_1\in \lambda$ such that for all $i\in D$, $d(b_i, a_i), d(a_i, b_{i+1})\leq \eta_1$. Then for any $i<j\in D$, by the transitivity of $d$, we have $d(a_i,a_j)\leq \max \{d(a_i, b_{i+1}), d(b_{i+1}, b_j), d(b_j, a_j)\}\leq \max\{\eta_0, \eta_1\}=_{def} \eta^*$ (here we use the convention that $d(t,t)=0$). Let $A'=\{a_i: i\in D\}$.

Pick $\delta\in A'$ such that $A'\cap \delta$ has size $\lambda$. We know $\sup d(\cdot, \delta)'' A'\cap \delta \leq \eta^*<\lambda$, which clearly implies there exist $\alpha_0<\alpha_1\in A'\cap \delta$ such that $d(\alpha_0, \delta)=d(\alpha_1,\delta)$. In particular, $A$ is not rainbow for $f$.
\end{proof}

\begin{remark}
$\square_\kappa$ implies $I[\kappa^+]$ is trivial. Hence  $\square_\kappa$ implies $\kappa^+\not\to^{poly}(cf(\kappa)^+)^2_{<\kappa-bdd}$.
\end{remark}

In light of the preceding theorems, the following theorem is the best possible in a sense.

\begin{definition}\label{covering}
A forcing poset $\mathbb{P}$ satisfies $<\kappa$-covering property if for any $\mathbb{P}$-name of subset of ordinals $\dot{B}$ such that $\Vdash_{\mathbb{P}} |\dot{B}|<\kappa$, there exists $B\in V$ such that $|B|<\kappa$ and $\Vdash_{\mathbb{P}} \dot{B}\subset B$.
\end{definition}

Notice that if $\kappa$ is singular, then $\kappa$ and $\kappa^+$ are preserved as cardinals in any forcing extension satisfying $<\kappa$-covering property.

\begin{theorem}\label{CountableCase}
Fix a singular cardinal $\kappa$ with $\lambda=\mathrm{cf}(\kappa)<\kappa$. Suppose $\kappa^{<\lambda}=\kappa$. Then for any $\alpha<\lambda^+$,  
\begin{equation}
\kappa^+\to^{poly} (\alpha)^2_{<\kappa-bdd}.
\end{equation}
Moreover, these partition relations continue to hold in any forcing extension by $\mathbb{P}$ satisfying the $<\kappa$-covering property.
\end{theorem}

\begin{proof}
We may assume $|\alpha|=\lambda$.
Fix a $\mathbb{P}$-name for a $<\kappa$-bounded coloring $\dot{f}$ on $[\kappa^+]^2$. We may assume it is normal.
Fix some large enough regular cardinal $\chi$. Build a sequence $\langle M_i \prec (H(\chi),\in, \dot{f}, \kappa, \mathbb{P}): i<\alpha\rangle$ such that 
\begin{enumerate}
\item $\kappa+1\subset M_i$, $|M_i|=\kappa$, $\kappa_i =_{def} M_i\cap \kappa^+ \in \kappa^+$,
\item $|\kappa_{i+1}-\kappa_i|=\kappa$,
\item ${}^{<\lambda} M_{i}\subset M_{i+1}$.
\end{enumerate}
The construction is possible since $\kappa^{<\lambda}=\kappa$. Fix a bijection $g: \lambda\to \alpha$. We will inductively define a rainbow subset $\{ a_i: i<\lambda\}$ such that $a_i\in \kappa_{g(i)+1}-\kappa_{g(i)}$. It is clear that this set as defined will have order type $\alpha$.
During the construction, we maintain the following \emph{construction invariant}: 

\emph{for any $i<\lambda$ and $l=g(i)$, whenever $a_j,a_k<\kappa_{l+1}$, we have $\Vdash_\mathbb{P} \dot{f}(a_j, \kappa_{l+1})\neq \dot{f}(a_k,\kappa_{l+1})$}.

Suppose for some $\beta<\lambda$ we have defined $A=\{ a_i: i<\beta\}$. Let  $l=g(\beta)$ and $B=\kappa_{l+1}-\kappa_{l}$. Our goal is to find an element in $B$ such that after we augment $A$ with this element, not only does the set remains a rainbow subset, but also the construction invariant is satisfied. Let $C=\{\delta<\kappa^+: \forall i,j<\beta \ a_i,a_j\in A\cap \kappa_{l+1}\rightarrow \Vdash_{\mathbb{P}} \dot{f}(a_i,\delta)\neq \dot{f}(a_j,\delta)\}$ and $B'=B\cap C$. 

\begin{claim}
$|B'|=\kappa$.
\end{claim}
\begin{proof}[Proof of the claim]
Let $A'=A\cap M_{l+1}=A\cap \kappa_{l+1}\subset M_l$. As ${}^{<\lambda}M_l\subset M_{l+1}$ we have $A'\in M_{l+1}$.
Hence  $C\in M_{l+1}$ and that $\kappa_{l+1}\in C$ by the construction invariant. $C$ is thus a stationary subset of $\kappa^+$. In particular, $M_{l+1}\models $ there exists an injection from $\kappa$ to $C$. As $\kappa+1\subset M_{l+1}$, $B\cap C=B'$ has size $\kappa$.
\end{proof}

We want to pick an element from $B'$ and add it to the set, however, we need to make sure the set is rainbow and satisfy the construction invariant. 
For any cardinal $\delta$, let $A\restriction \delta$ be $A\cap (<\delta)$.
For the purpose of presentation, work in $V[G]$ for some $G\subset \mathbb{P}$ generic over $V$.

Let $B_{-1}=\{\delta\in B': \exists a\in A\restriction \kappa_{l+1} \ f(\delta,\kappa_{l+1})=f(a,\kappa_{l+1})\}$. For each $i<\beta$ with $g(i)>l$, let $B_i=\{\delta\in B': \exists \alpha\in A\restriction \kappa_{g(i)+1} \ f(\alpha,\kappa_{g(i)+1})=f(\delta,\kappa_{g(i)+1})\}$ and $B'_i=\{\delta\in B': \exists \alpha\in A\restriction a_i \ f(\alpha, a_i)=f(\delta,a_i)\}$. We verify that these sets as defined all have size $<\kappa$.

Suppose for the sake of contradiction that $B_{-1}$ has size $\kappa$, then since $|A|<\kappa$ and $|B'|=\kappa$, there exists $a\in A$ such that $\{\delta\in B': f(a,\kappa_{l+1})=f(\delta,\kappa_{l+1})\}$ has size $\kappa$. This contradicts with the assumption that $f$ is $<\kappa$-bounded.

Suppose for the sake of contradiction that for some $i$ with $i<\beta$ and $g(i)>l$ we have $|B_i|=\kappa$, similar to the above, we can find $a\in A$ such that $\{\delta\in B': f(a, \kappa_{g(i)+1})=f(\delta,\kappa_{g(i)+1})\}$ has size $\kappa$, contradicting with $<\kappa$-boundedness. Similarly $|B_i'|<\kappa$.

Back in $V$, pick $\mathbb{P}$-names for the sets above: $\dot{B}_{-1}$, $\dot{B}_i, \dot{B}'_i$ for all $i<\beta$ such that $g(i)>l$. By the $<\kappa$-covering property of $\mathbb{P}$, we can find $B_{-1}^*, B_i^*, (B_i')^*$ of size $<\kappa$ in $V$ such that $\Vdash_{\mathbb{P}} \dot{B}_{-1}\subset B_{-1}^*, \dot{B}_i\subset B_i^*, \dot{B}'_i\subset (B'_i)^*$ for all $i<\beta$ with $g(i)>l$. Since $\beta<\lambda=cf(\kappa)$, we know $|B_{-1}^*\cup \bigcup_{i<\beta, g(i)>l} B_i^*\cup (B_i')^*|<\kappa$.
Pick $a_\beta\in B'-B_{-1}^*-\bigcup_{i<\beta, g(i)>l} B_i^*\cup (B_i')^*$. Then it follows that $A\cup \{a_\beta\}$ is forced by $\mathbb{P}$ to be a rainbow subset and to satisfy the construction invariant.

\end{proof}

An immediate consequence of the proof of Theorem \ref{CountableCase} is: 

\begin{cor}
For any cardinal $\kappa$ and any $\alpha<\omega_1$, 
\begin{equation}
\kappa^+\to^{poly} (\alpha)^2_{<\kappa-bdd}.
\end{equation}
\end{cor}

\begin{question}
Is $\kappa^+\to^{poly} (\omega_1)^2_{<\kappa-bdd}$ consistent for some singular $\kappa$ of countable cofinality?
\end{question}

\section{A coloring that is strongly proper indestructible but c.c.c destructible}\label{indestructible}

It is proved in \cite{MR2354904} that if $CH$ holds, then $\omega_2\to^{poly} (\eta)^2_{<\omega_1-bdd}$ for any $\eta<\omega_2$. In \cite{MR2902230}, a model where $2^\omega=\omega_2$ and $\omega_2\not\to^{poly} (\omega_1)_{2-bdd}^2$ is constructed. A question regarding the possibility of getting $\omega_2\not\to^{poly} (\omega_1)_{2-bdd}^2$ along with continuum larger than $\omega_2$ was raised. A positive answer was given in \cite{MR3437648} using the method of forcing with symmetric systems of submodels as side conditions.

In this section we give a simpliflied construction of the model presented in \cite{MR2902230} using the framework developed by Neeman \cite{MR3201836} and show the witness to $\omega_2\not\to^{poly} (\omega_1)_{2-bdd}^2$ in that model is indestructible under strongly proper forcings. This provides an alternative answer to the original question.

\begin{definition}[Special case of Definition 2.2 and 2.4 in \cite{MR3201836}]
Let $K=(H(\omega_2), <^*)$ where $<^*$ is some well-ordering of $H(\omega_2)$.
Define \emph{small nodes} and \emph{transitive nodes} respectively as 
$$\mathcal{S}=_{def}\{M \in [K]^\omega:  M\prec K\}$$ and
 $$\mathcal{T}=_{def}\{W\in [K]^{\omega_1} : W\prec K\text{ and internally approachable of length }\omega_1\}.$$ Here $W\prec K$ is \emph{internally approachable of length $\omega_1$} if there exists a continuous $\subseteq$-increasing countable sequence $\langle W_i \prec K: i<\omega_1\rangle$ such that $W=\bigcup_{i<\omega_1} W_i$ and for all $i<\omega_1$, $\langle W_j: j\leq i\rangle\in W_{i+1}$.
 
Both sets are stationary in $K$ respectively.
Let $\mathbb{P}=\mathbb{P}_{\omega,\omega_1,\mathcal{S}, \mathcal{T}}$ be the standard sequence poset consisting of models of two types. More precisely, $\mathbb{P}$ consists of finite increasing $\in$-chain of elements in $\mathcal{S}\cup \mathcal{T}$ closed under intersection. For example, a typical element will look like $\{s_0, s_1, \cdots, s_{k-1}\}\subset \mathcal{S}\cup \mathcal{T}$, where for any $i<k-1$, $s_i\in s_{i+1}$ and for any $i,j<k$, there is some $l<k$ satisfying $s_i\cap s_j=s_l$. Notice that we can either think of a condition as a finite sequence or as a finite set, since the elements in a condition can be naturally ordered by their Von Neumann ranks. Thus, given a condition $s$ and $M, M'\in s$, we say $M$ \emph{precedes/is before (succeeds/is after)} $M'$ when the rank of $M$ is smaller (greater) than the rank of $M'$.
\end{definition}

\begin{remark}\label{elaborate}
In order to consolidate the reader's understanding of the notion, we point out the following: 
\begin{enumerate}
\item If $M_0\in \mathcal{S}, M_1\in \mathcal{S}\cup \mathcal{T}$ and $M_0\in M_1$, then $M_0\subset M_1$. However, if $W\in \mathcal{T}$ and $M\in \mathcal{S}$ satisfy $W\in M$, it cannot be the case that $W\subset M$. Hence, in general the membership relation $\in$ restricted on a condition in $\mathbb{P}$ is not transitive.
\item If $W\in \mathcal{T}$ and $M\in \mathcal{S}$ satisfy that $W\in M$, then $W\cap M\in \mathcal{S}\cap W$. Let $\langle W_i: i<\omega\rangle$ be the $<^*$-least sequence witnessing that $W$ is internally approachable of length $\omega_1$. Let $M\cap \omega_1=\delta$. We claim that $M\cap W=W_\delta$. On the one hand, we have $\langle W_i: i<\omega\rangle\in M$ which implies $\bigcup_{i<\delta}W_i= W_\delta\subset M$. On the other hand, suppose $x\in M\cap W$, then if $i_x\in \omega_1$ is the least such that $x\in W_{i_x}$, we know $i_x\in M$ by elementarity. Hence $i_x\in \delta$, which implies $x\in W_\delta$.
\item If $s\in \mathbb{P}$ and $W\in s\cap \mathcal{T}$, then any $M\in s$ preceding $W$, $M\in W$. If $M_0, M_1\in s\cap \mathcal{S}$ such that there is no transitive node in $s$ between $M_0$ and $M_1$, then $M_0\in M_1$.
\end{enumerate}
\end{remark}

It is not necessary for a reader to be familiar with \cite{MR3201836} in order to understand the following proof since we will list all the lemmas needed.

\begin{claim}[Claim 2.17, 2.18]\label{residuegap}
Fix $s\in \mathbb{P}$ and $Q\in s$. Define $res_Q(s)=s\cap Q$. Then \begin{enumerate}
\item $res_Q(s)\in \mathbb{P}$.
\item If $Q$ is a transitive node, then $res_Q(s)$ consists of all nodes of $s$ that occur before $Q$. If $Q$ is a small node, then $res_Q(s)$ consists of all nodes in $s$ that occur before $Q$ and do not belong to any interval $[Q\cap W, W)\cap s$ for any transitive node $W\in s$. Those intervals are called \emph{residue gaps} of $s$ in $Q$.
\end{enumerate}

\end{claim}

\begin{remark}
Do not confuse $res_Q(s)$ with the set of nodes in $s$ preceding $Q$. The point of (1) is that the part of information about $s$ that is captured by $Q$ is itself a legitimate condition. The fact that $res_Q(s)$ is closed under intersection is immediate since $s$ is. It takes a little work to show it forms an $\in$-increasing chain.

The second part of the claim describes what $res_Q(s)$ looks like in a very concrete way. For any $W\in \mathcal{T}\cap s$, $[Q\cap W, W)\cap s \cap Q$ must be empty. It takes more work to show that any $M\in s-Q$, there exists one such residue gap containing $M$.
\end{remark}

\begin{lemma}[Corollary 2.31 in \cite{MR3201836}]\label{strongproperness}
Let $s\in \mathbb{P}$ and $Q\in s$. For any $t\in \mathbb{P}\cap Q$ such that $t\leq res_Q(s)=_{def}s\cap Q\in \mathbb{P}$. Then 
\begin{enumerate}
\item $s$ and $t$ are directly compatible, namely the closure of $s\cup t$ under intersection is a common lower bound. Moreover, if $Q$ is a transitive node, then $s\cup t$ is already closed under intersection hence is the lower bound for $s$ and $t$.
\item If $r$ is the closure of $s\cup t$, then $res_Q(r)=t$.
\end{enumerate}
\end{lemma}

\begin{remark}
An immediate corollary is that that if $Q\in \mathcal{S}$ contains $s\in Q\cap \mathbb{Q}$, then the closure of $s\cup \{M\}$ under intersection is a greatest lower bound of $s$ containing $Q$. The basic idea of the proof is that: first verify that $s\cup t$ consists of $\in$-increasing nodes, and then show that this property remains even after we add nodes of the form $M_0\cap M_1$ where $M_0\in s, M_1\in t$.
\end{remark}

For each $\beta<\omega_2$, let $f_\beta$ be the $<^*$-least injection from $\beta$ to $\omega_1$. Define the main forcing $\mathbb{Q}$ to consist of $p=(c_p, s_p)$ such that: 
\begin{enumerate}
\item $c_p$ is a finite partial function from $[\omega_2]^2\to \omega_1$ satisfying the \emph{bounding requirement}, namely there do not exist $\alpha_0<\alpha_1<\alpha_2<\beta$ such that $(\alpha_i,\beta)\in dom(c_p)$ for all $i<3$ and $c_p(\alpha_0,\beta)=c_p(\alpha_1,\beta)=c_p(\alpha_2,\beta)$;
\item for any $(\alpha,\beta)\in dom(c_p)$, $c_p(\alpha,\beta)\geq f_\beta(\alpha)$;
\item $s_p\in \mathbb{P}$;
\item if $(\alpha,\beta)\in dom(c_p)$ and $M\in s_p$ contains $(\alpha,\beta)$, then $c_p(\alpha,\beta)\in M$.
\end{enumerate}

$q\leq_{\mathbb{Q}} p$ iff $c_q\restriction dom(c_p)=c_p$ and $s_q \supset s_p$.

\begin{claim}\label{enlarge}
For any $\alpha<\beta<\omega_2$ and $p\in \mathbb{Q}$, there exists $p'\leq p$ such that $(\alpha,\beta)\in dom(c_{p'})$.
\end{claim}
\begin{proof}
We may assume $(\alpha,\beta)\not \in dom(c_p)$. Let $\delta\leq\omega_1$ be the least such that there is some $M\in s\cap \mathcal{S}$ containing $(\alpha,\beta)$ such that $\delta=M\cap \omega_1$ if it exists, $\delta=\omega_1$ if no such node exists. In any case, we have $\delta$ is a limit ordinal and $f_\beta(\alpha)\in \delta$. Pick $\delta\backslash (f_\beta(\alpha)+1)$ which is not in $range(c_p)$. It is clear that $(c_p\cup (\{\alpha,\beta\},\gamma), s_p)$ is a desired extension. 
\end{proof}

\begin{definition}
Let $\lambda$ be a fixed regular cardinal, $P$ be a poset. Let $\mathcal{M}=(H(\lambda), \in , \cdots)$ be some countable extension of $(H(\lambda), \in )$. We say $P$ is strongly proper for $B$ where $B\subset \{M: M\prec \mathcal{M}\}$ if for any $M\in B$ and any $r\in M\cap P$, there exists $r'\leq r$ such that $r'$ is \emph{strongly $(M,P)$-generic}, namely for any $r''\leq r'$, there exists some $r^* \in M\cap P$ such that any $t\leq r^*$ with $t\in M$ is compatible with $r''$. We call such $r^*$ a \emph{reduct} of $r''$ on $M$ and for the rest of the section we will use $r''\restriction M$ to represent one such reduct of $r''$ on $M$.

$P$ is strongly proper if for all sufficiently large regular $\theta$, $P$ is strongly proper for a club subset of $\{M\in [H(\theta)]^\omega: M\prec H(\theta)\}$.
\end{definition}

\begin{claim}\label{StrongPropernessTransitive}
For any $p=(c_p, s_p)$ with a transitive node $W\in s_p$, if $t\leq (c_p\cap W, res_W(s_p))$ and $t\in W$, then $t$ and $p$ are compatible. In particular, $\mathbb{Q}$ is strongly proper for $\mathcal{T}$. 
\end{claim}
\begin{proof}
Implicitly in the statement of the claim, $(c_p\cap W, res_W(s_p))$ can be easily checked to be a condition.
It is left to check that $r=(c_t\cup c_p, s_p\cup s_t)$ is a condition since it is clear that it extends $t$ and $p$. First note that $s_p\cup s_t \in \mathbb{P}$ and $\leq_\mathbb{P}$-extends $s_p$ and $s_t$ by Lemma \ref{strongproperness}. It is also clear that $c_t\cup c_p$ is a function that satisfies the bounding requirement. We are left with checking condition (4) as in the definition of $\mathbb{Q}$. Given $(\alpha,\beta)\in dom(c_r)$ and $M\in s_r$, if $(\alpha,\beta)\in M$, we need to show $c_r(\alpha,\beta)\in M$. Since $t$ and $p$ are conditions, the following cases are what we need to check: 
\begin{itemize}
\item $(\alpha,\beta)\in dom(c_p)-dom(c_t)$ and $M\in s_t-s_p$,
\item $(\alpha,\beta)\in dom(c_t)-dom(c_p)$ and $M\in s_p-s_t$.
\end{itemize}
If $(\alpha,\beta)\in dom(c_p)-dom(c_t)$, $(\alpha,\beta)\not \in W$ so $(\alpha,\beta)\not \in M$ for any $M\in s_t$ as $t\in W$ and $W$ is transitive. If $(\alpha,\beta)\in dom(c_t)-dom(c_p)$ and $M\in s_p-s_t$ containing $(\alpha,\beta)$, then $M\cap W\in s_p\cap W\subset s_t$. As $t$ is a condition, we have $c_t(\alpha,\beta)\in M\cap W\subset M$.

To see $\mathbb{Q}$ is strongly proper for $\mathcal{T}$, it suffices to notice that for any $W\in \mathcal{T}$ and $t=(c_t,s_t)\in W\cap \mathbb{Q}$, there exists $t'=(c_t, s_t')\leq t$ such that $W\in s_t'$ by Lemma \ref{strongproperness}.
\end{proof}

\begin{claim}\label{StrongPropernessCountable}
 For any countable $M^*\prec H(\lambda)$ for some large enough regular $\lambda$ containing $\mathbb{Q}, K$, if $r\in \mathbb{Q}$ satisfies that $M^*\cap K\in s_r$, then $r$ is strongly $(M^*, \mathbb{Q})$-generic. In particular, $\mathbb{Q}$ is strongly proper.
\end{claim}

\begin{proof}
Let $M=M^*\cap K$. We need to show for any $r'\leq r$, there exists $r'\restriction M\in M\cap \mathbb{Q}$ weaker than $r'$, such that any extension of $r'\restriction M$ in $M$ is compatible with $r'$. Let $r'\restriction M$ be $(c_{r'} \cap M,  res_M(s_{r'}))$. It is easy to see that $r'\restriction M$ is a condition weaker than $r'$.

Let $t\in \mathbb{Q}\cap M$ be such that such that $t\leq r'\restriction M$. As $s_t\leq res_M(s_{r'})$ and $s_t\in M$, we know by Lemma \ref{strongproperness} there exists $s^*\leq s_t, s_{r'}$ such that $res_M(s^*)=s_t$. Furthermore, we may assume $s^*$ is the closure of $s_{r'}\cup s_t$ under intersection. Let $h=_{def}(c_t\cup c_{r'}, s^*)$. We will check that $h$ is a condition.

First we check that $c_t\cup c_{r'}$ is a function that satisfies the bounding requirement. To see it is a function, let $(\alpha,\beta)\in dom(c_t)\cap dom(c_{r'})$, then $(\alpha,\beta)\in M$. Since $c_t\supset c_{r'}\restriction M$, we know $c_t(\alpha,\beta)=c_{r'}(\alpha,\beta)$. To see $c_t\cup c_{r'}$ is 2-bounded, suppose for the sake of contradiction, $\alpha_0<\alpha_1<\alpha_2<\beta$ are such that $c_h(\alpha_0,\beta)=c_h(\alpha_1,\beta)=c_h(\alpha_2,\beta)=\gamma\in \omega_1$. Note that there exists some $i<3$ such that $(\alpha_i, \beta)\in M$ since otherwise $(\alpha_k,\beta)\in dom(c_{r'})$ for all $k<3$, which contradicts with the fact that $r'$ is a condition. Also notice that $c_t(\alpha_i,\beta)=\gamma\in M$. By the requirement of a condition we know $f_\beta (\alpha_j)\leq \gamma$ for all $j<3$. But as $\gamma\in M$, $\gamma\subset M$, we know $\alpha_j\in M$ for all $j<3$. This means these three tuples are all in the domain of $c_t$. This is a contradiction to the fact that $t$ is a condition.

Finally we check condition (4) in the definition of $\mathbb{Q}$. Given $(\alpha,\beta)\in dom(h)$ and $N\in s_h$, if $(\alpha,\beta)\in N$, then we need to verify $c_h(\alpha,\beta)\in N$.
Recall that each element of $s_h$ is of the form $M_0\cap M_1$, $M_0$ or $M_1$ where $M_0\in s_{r'}, M_1\in s_t$. Hence, since $r'$ and $t$ are conditions, the cases we need to verify are: 
\begin{itemize}
\item $(\alpha,\beta)\in dom(c_{r'})-dom(c_{t})$ and $N\in s_{t}-s_{r'}$ and 
\item $(\alpha,\beta)\in dom(c_t)-dom(c_{r'})$ and $N\in s_{r'}-s_t$.
\end{itemize}

For $(\alpha,\beta)\in dom(c_{r'})-dom(c_t)$ and $N\in s_t\cap \mathcal{S}-s_{r'}$, we know $(\alpha,\beta)\not\in M$ since otherwise, it would have been in $dom(c_{r'}\restriction M)\subset dom(c_t)$. But $s_t\in M$ since $t\in M$, which implies $N\in M$. Hence it is impossible to have $(\alpha,\beta)\in N$.

For $(\alpha,\beta)\in dom(c_t)-dom(c_{r'})$ and $N\in s_{r'}\cap \mathcal{S}-s_t$ such that $(\alpha,\beta)\in N$, since $(\alpha,\beta)\in M$ and $s_{r'}$ is closed under intersection, we may assume $N\subset M$. If $N=M$, then we are done since $c_h(\alpha,\beta)=c_{t}(\alpha,\beta)\in M$. If $N\in M$, then we are done since $t\leq r'\restriction M$. So assume $N\not \in M$. By Claim \ref{residuegap}, $N$ occurs in a residue gap, namely there exists $W\in M\cap s_{r'}$ such that $N\in [W\cap M, W)=_{def} \{M'\in s_{r'}: rank(W\cap M)\leq rank(M') < rank(W)\}$. We will show $c_h(\alpha,\beta)\in N$ by inducting on the rank of $N$.
As $(\alpha,\beta)\in N\subset W$, $(\alpha,\beta)\in M\cap W$. Also $c_h(\alpha,\beta)\in M\cap W$. If there is no transitive node between $W\cap M$ and $N$, then we are done since $W\cap M\subset N$ (recall that $s_{r'}$ is linearly ordered by $\in$ and Remark \ref{elaborate}). Otherwise, there exists $W'\in [W\cap M, N)\cap \mathcal{T}$. Let $N'= W'\cap N$. Then $rank(N')<rank(N)$. Since $(\alpha,\beta)\in N'$, by the induction hypothesis, we know that $c_h(\alpha,\beta)\in N' \subset N$.

To see $\mathbb{Q}$ is strongly proper, for any condition $p$, for sufficiently large regular cardinal $\lambda$, we can find $M^*\prec H(\lambda)$ containing $p, K, \mathbb{Q}$. Then $p'=(c_p, cl(s_p\cup \{M^*\cap K\}))$ is a strongly $(M^*, \mathbb{Q})$-generic extension of $p$ by Lemma \ref{strongproperness}, where $cl(s_p\cup \{M^*\cap K\})$ denotes the closure of $s_p\cup \{M^*\cap K\}$ by intersection.

\end{proof}

By Claim \ref{StrongPropernessCountable} and Claim \ref{StrongPropernessTransitive}, $\omega_1$ and $\omega_2$ are preserved in the forcing extension by $\mathbb{Q}$.

\begin{lemma}[Lemma 4.3 of \cite{MR2902230}]\label{borrow}
For $\alpha_0<\alpha_1<\beta<\omega_2$ and $p\in \mathbb{Q}$, if $(\alpha_i,\beta)\not \in dom(c_p)$ for any $i<2$ and 

\begin{equation}
\forall M\in s_p \ (\alpha_0,\beta)\in M \Leftrightarrow (\alpha_1,\beta)\in M
\end{equation}
Then there exists an extension $p'=(c_{p'}, s_p)$ with the same side condition such that $(\alpha_0,\beta), (\alpha_1,\beta)\in dom(c_{p'})$ and $c_{p'}(\alpha_0,\beta)=c_{p'}(\alpha_1,\beta)$. Furthermore, we can ensure that $dom(c_{p'})=dom(c_p)\cup \{(\alpha_0,\beta),(\alpha_1,\beta)\}$.

\end{lemma}

Building on the idea of Lemma 4.6 in \cite{MR2902230}, we prove a strengthened version in the following.

\begin{lemma}\label{strengthen}
In $V^{\mathbb{Q}}$, for any strongly proper forcing $\dot{P}$, $\Vdash_{\dot{P}}$ $c$ witnesses $\omega_2^V\not\to^{poly} (\omega_1)^2_{2-bdd}$.
\end{lemma}

\begin{remark}\label{confuse}
More accurately, it is the coloring $c': [\omega_2]^2\to \omega_2$ such that $c'(\alpha,\beta)=(c(\alpha,\beta), \beta)$ that witnesses $\omega_2\not\to^{poly} (\omega_1)^2_{2-bdd}$. As it is clear from the context, we will continue to refer to $c$ as the witness in the following.
\end{remark}

\begin{proof}[Proof of Lemma \ref{strengthen}]
Suppose otherwise for the sake of contradiction. Let $r\in \mathbb{Q}$, $\mathbb{Q}$-name $\dot{p}, \dot{P}$, $\mathbb{Q}*\dot{P}$-name $\dot{X}$, $\gamma\in \omega^V_2+1$ such that 
\begin{enumerate}
\item $r\Vdash_{\mathbb{Q}} \dot{P}$ is a strongly proper forcing and $\dot{p}\in \dot{P}$ and 
\item $r\Vdash_{\mathbb{Q}} \dot{p}\Vdash_{\dot{P}} \sup \dot{X}=\gamma, \dot{X}$ is a rainbow subset for $c$ of order type $\omega_1$.
\end{enumerate}

Note that we include the possibility that $\gamma=\omega_2^V$ since it may be collapsed by $\mathbb{Q}*\dot{P}$. In either case, $cf(\gamma)>\omega$.

Let $G\subset \mathbb{Q}$ containing $r$ be generic over $V$. Fix some sufficiently large regular cardinal $\lambda$ and let $C=(\dot{C})^G\subset ([H(\lambda)]^\omega)^{V[G]}$ be a club that witnesses the strong properness of $P$ in $V[G]$.

\begin{claim}
For any stationary subset $T\subset [H(\lambda)]^\omega$ in $V$, $T[G]=_{def} \{M[G]: M\in T\}$ is a stationary subset of $([H(\lambda)]^\omega)^{V[G]}$.
\end{claim}

\begin{proof}[Proof of the claim]
In $V[G]$, let $f: H(\lambda)^{<\omega} \to H(\lambda)$. In $V$, let $\lambda^*$ be much larger regular cardinal than $\lambda$ and $M'\prec H(\lambda^*)$ containing $\dot{f}, H(\lambda)$ be such that $M=M'\cap H(\lambda)\in T$. Then $M[G]$ is closed under $f$, since for any $\bar{a}\in M[G]\cap [H(\lambda)^{V[G]}]^{<\omega} $, $f(a)\in M'[G]\cap (H(\lambda))^{V[G]} =M'[G]\cap H(\lambda)[G]=(M'\cap H(\lambda))[G]$. The last equality holds since for any $\dot{\tau}\in M', \dot{\sigma}\in H(\lambda)$ such that $(\dot{\tau})^G=(\dot{\sigma})^G$, by the fact that $M'[G]\prec H(\lambda^*)[G]$, $M'[G]\models $ there exists $\dot{\sigma}\in H(\lambda)^V$, $\dot{\tau}^G=\dot{\sigma}^G$. It is easy to see this is sufficient since $M'[G]\cap H(\lambda)^V=M'\cap H(\lambda)^V$.
\end{proof}

Find a countable $N'\in V$ such that $N'\prec H(\lambda)^V$ contains $r,\mathbb{Q}, \dot{p}, \dot{P}, \dot{X},\gamma$. Moreover, $N=_{def} N' \cap K \in \mathcal{S}$ and $N'[G]\in C$.

Let $\gamma'=\sup N\cap \gamma$. Extend $r$ to $t$ such that $N\in s_t$ by Lemma \ref{strongproperness}. Consequently, $t$ is strongly $(N', \mathbb{Q})$-generic. Find $t'\leq_{\mathbb{Q}} t$, $\beta\in [\gamma', \gamma)$ and $\mathbb{Q}$-names $\dot{p}'$, $\dot{q}$ such that $\dot{q} \in N'$ and $t'\Vdash_{\mathbb{Q}} \dot{p}'$ is strongly $(N'[\dot{G}], \dot{P})$-generic and $\dot{p}'\leq_{\dot{P}}\dot{p}$, $\dot{p}'\restriction N'[\dot{G}]=\dot{q}$ and $\dot{p}'\Vdash_{\dot{P}} \beta\in \dot{X}$. Let $m=|t'|<\omega$.

Now consider $D=\{a\leq_{\mathbb{Q}} t'\restriction N': \exists \dot{b} \  a\Vdash_{\mathbb{Q}} \dot{b}\leq_{\dot{P}} \dot{q}, \exists \alpha_0<\cdots <\alpha_{2^m} \ \dot{b}\Vdash_{\dot{P}} \forall i\leq 2^m \ \alpha_i\in \dot{X}\}$. This set is dense below $t'\restriction N'$ and is in $N'$. Pick $a\in D\cap N'$ and $\dot{b}, \alpha_0,\cdots, \alpha_{2^m} \in N'$ as its witness. By the Pigeonhole principle, there exist $i\neq j\leq 2^m$ such that for any $M \in  \mathcal{S}\cap s_{t'}$, $(\alpha_i,\beta)\in M$ iff $(\alpha_j,\beta)\in M$. Apply Lemma \ref{borrow}, there exists $t''\leq t'$ such that $c_{t''}(\alpha_i,\beta)=c_{t''}(\alpha_j,\beta)$ with $s_{t''}=s_{t'}$ and $dom(c_{t''})=dom(c_{t'})\cup \{(\alpha_i,\beta), (\alpha_j,\beta)\}$.
As $a\leq_{\mathbb{Q}} t'\restriction N'=t''\restriction N'$, $a$ and $t''$ are compatible. Find a common lower bound $t'''\leq_{\mathbb{Q}} a, t''$. Then $t'''\Vdash_{\mathbb{Q}}\dot{b}\leq_{\dot{P}} \dot{q}=\dot{p}'\restriction N'[\dot{G}]$ and $\dot{b}\in N'[\dot{G}]$. Hence $t'''$ forces $\dot{b}$ and $\dot{p}'$ are compatible. Let $\dot{w}$ be a common lower bound. Then $(t''', \dot{w})$ forces $c(\alpha_i,\beta)=c(\alpha_j,\beta)$ as well as $\alpha_i, \alpha_j, \beta\in \dot{X}$. This is a contradiction since $(t''',\dot{w})\leq_{\mathbb{Q}*\dot{P}} (r,\dot{p})$ and  $(r,\dot{p}) \Vdash _{\mathbb{Q}*\dot{P}} \dot{X}$ is a rainbow subset for $c$.

\end{proof}

An immediate consequence is $\omega_2\not\to^{poly} (\omega_1)^2_{2-bdd}$ is consistent with the continuum being arbitrarily large as Cohen forcings are strongly proper. This provides an alternative answer to a question in \cite{MR2902230}, which was originally answered in \cite{MR3437648} using a different method.

However In this model, there exists a c.c.c forcing that forces a rainbow subset into $c\restriction [\omega_1]^2$. In $V^{\mathbb{Q}}$, let $R$ be the poset $\{a\in [\omega_1]^{<\omega}: a \text{ is a rainbow subset for }c\}$ order by reverse inclusion. By Remark \ref{confuse}, $a\in [\omega_1]^{<\omega}$ is a rainbow subset for $c$ if there is no $\alpha_0<\alpha_1<\beta \in c$ such that $c(\alpha_0,\beta)=c(\alpha_1,\beta)$.
It is easy to see that in $V^{\mathbb{Q}}$, $R$ adds an unbounded subset of $\omega_1^V$.

\begin{lemma}
In $V^{\mathbb{Q}}$, $R$ is c.c.c.
\end{lemma}

\begin{proof}
Otherwise, let $\langle \dot{\tau}_i: i<\omega\rangle$ be a head-tail-tail system with root $r\in [\omega_1]^{<\omega}$ that is forced to be an uncountable antichain by $p$.
Let $N'\prec H(\lambda)$ contain relevant objects for some sufficiently large regular cardinal $\lambda$. Let $\delta=N'\cap \omega_1$. Let $q\leq p$ be a strongly $(N',\mathbb{Q})$-generic condition that determines some $\dot{\tau}_j=h$ such that $\min (h-r)\geq \delta$. Let $q'=q\restriction N'$. Find $t\leq q'$ in $N'$ such that $t$ decides some $\dot{\tau}_i = h' \in N'$ such that $\min (h'-r)\geq \max_{(\alpha,\beta)\in dom(c_q)\cap N'} \max \{\alpha,\beta\} +1$. Now we extend $q$ to $q^*$ such that $s_{q}=s_{q^*}$ and $dom(c_{q^*})$ includes $h'\times h$ such that $c_{q^*} [(h'-r)\times (h-r)] \cap (\delta\cup range(c_q))=\emptyset$, $c_{q^*}\restriction (h'-r)\times (h-r)$ is injective and $q^*\restriction N' = q'$. To see that we can do this, enumerate $(h'-r)\times (h-r)$ as $\{(\alpha_i,\beta_i): i<k\}$. We inductively add $(\alpha_i,\beta_i)$ to $c_q$ by Claim \ref{enlarge} while maintaining the other requirements. More precisely, suppose we have added $(\alpha_j,\beta_j)$ to the domain of $c_p$ for $j<i$. Let $M\in s_p$ be of the minimum rank such that $(\alpha_i,\beta_i)\in M$. Then $M\cap \omega_1>\max\{\delta, f_{\beta_i}(\alpha_i)\}$. Hence we only need to avoid finitely many elements in $M\cap \omega_1-(\max\{\delta, f_{\beta_i}(\alpha_i)\}+1)$, which is clearly possible. $q^*$ is compatible with $t$ since $t\leq q'=q^*\restriction N'$ and $q^*\leq q$ which is strongly $(N',\mathbb{Q})$-generic. But a common extension of $q^*$ and $t$ forces that $\dot{\tau}_i \cup \dot{\tau}_j$ is rainbow. We have reached the desired contradiction.
\end{proof}

\begin{remark}
Similarly, if we force with finite rainbow subsets of $\omega_2$, then we will add a rainbow subset of size $\aleph_2$ for $c$ in a c.c.c forcing extension.
\end{remark}


\section{Some remarks and questions on partition relations of triples}\label{generalization}

Recall that Todorcevic in \cite{MR716846} showed that it is consistent that $\omega_1\to^{poly} (\omega_1)_{<\omega-bdd}^2$. In fact, he showed a stronger conclusion, namely for any $<\omega$-bounded coloring on $[\omega_1]^2$, it is always possible to partition $\omega_1$ into countably many rainbow subsets. The plain generalization of this result to 3-dimensional case fails miserably.

\begin{remark}
$\omega_1\not \to^{poly} (4)^3_{<\omega-bdd}$.
\end{remark}

\begin{proof}
Fix $a: [\omega_1]^2\to \omega$ such that for each $\alpha<\omega_1$, $a(\cdot, \alpha)$ is an injection from $\alpha$ to $\omega$. Define $f: [\omega_1]^3\to \omega$ such that $\{\alpha,\beta,\gamma\}_{<}$ is defined to be $\max \{a(\alpha,\gamma), a(\beta,\gamma)\}\in \omega$. Now define $g: [\omega_1]^3\to \omega_1$ to be $g(\{\alpha,\beta,\gamma\})=(f(\{\alpha,\beta,\gamma\}),\gamma)$. Note $g$ is $<\omega$-bounded, since for each $\gamma\in \omega$, there are only finitely many $\alpha<\gamma$ such that $a(\alpha,\gamma)<n$. For any $A=\{\alpha_0<\alpha_1<\alpha_2<\alpha_3\}\subset \omega_1$ of size 4, pick $i<3$ such that for any $j<3$ and $j\neq i$, $a(\alpha_j,\alpha_3)<a(\alpha_i,\alpha_3)=n$. Say $i=0$ for the sake of demonstration. Then $\{\alpha_0, \alpha_1,\alpha_3\}$ and $\{\alpha_0, \alpha_2,\alpha_3\}$ get the same color $(n,\gamma)$.
\end{proof}

\begin{remark}
There are various limitations on Ramsey Theorems for higher dimensions. For example, $2^\omega \not\to (\omega+2)^3_2$. Hence we need other methods to prove higher dimensional rainbow Ramsey theorems.
\end{remark}

Given a 2-bounded normal coloring $f$ on $[\delta]^3$, let us try to classify what types of obstacles there are for getting a rainbow subset.

\begin{enumerate}
\item[Type 1] for some $\alpha,\beta,\alpha',\beta'<\gamma$ such that $\{\alpha,\beta\}\cap \{\alpha',\beta'\}=\emptyset$ and $f(\alpha,\beta,\gamma)=f(\alpha',\beta',\gamma)$
\item[Type 2] for some $\alpha<\beta<\gamma<\delta$, $f(\alpha,\gamma, \delta)=f(\alpha,\beta, \delta)$
\item[Type 3] for some $\alpha<\beta<\gamma<\delta$, 
$f(\alpha,\beta,\delta)=f(\beta,\gamma,\delta)$
\item[Type 4] for some $\alpha<\beta<\gamma<\delta$, 
$f(\alpha,\gamma,\delta)=f(\beta,\gamma,\delta)$.
\end{enumerate}

\begin{remark}
By repeatedly applying the Ramsey theorem on $\omega$ to eliminate bad tuples of the 4 types above, one can show $\omega_1\to^{poly} (\omega+k)^3_{l-bdd}$ for any $k,l\in \omega$. This is already in contrast with the dual statements in Ramsey theory.
\end{remark}

\begin{question}
Can we prove in ZFC that $\omega_1\to^{poly} (\alpha)^3_{2-bdd}$ for any $\alpha<\omega_1$?
\end{question}

\begin{question}
Is $\omega_1\to^{poly} (\omega_1)^3_{2-bdd}$ consistent? Is it a consequence of $\mathrm{PFA}$?
\end{question}

\bibliographystyle{plain}
\bibliography{bib}

\Addresses

\end{document}